\documentclass{amsart}
\usepackage{amsmath}
\usepackage{amssymb}
\usepackage{array}
\usepackage{amscd}
\usepackage{tikz-cd}

\newtheorem{Th}{Theorem}
\newtheorem{Lemma}{Lemma}
\newtheorem{Prop}{Proposition}

\newtheorem{Cor}{Corollary}
\theoremstyle{definition}
\newtheorem{definition}{Definition}
\newtheorem*{definition*}{Definition}
\newtheorem{Remark}{Remark}
\newtheorem{Example}{Example}
\newcommand{\comm}[1]{}

\usepackage{color}
\usepackage[normalem]{ulem}
\definecolor{DarkGreen}{rgb}{0,0.5,0.1} 

\newcommand\soutD{\bgroup\markoverwith
{\textcolor{DarkGreen}{\rule[.5ex]{2pt}{1pt}}}\ULon}

\makeatletter
\newcommand*{\rom}[1]{\expandafter\@slowromancap\romannumeral #1@}
\makeatother

\begin{document}

\title{Jordan constant for Cremona group of rank 2 over a finite field}

\author{Anastasia V.Vikulova}
\address{{\sloppy
\parbox{0.9\textwidth}{
Laboratory of Algebraic Geometry, National Research University Higher
School of Economics, 6 Usacheva str., Moscow, 119048, Russia.
\\[5pt]
Steklov Mathematical Institute of Russian
Academy of Sciences,
8 Gubkin str., Moscow, 119991, Russia
}\bigskip}}
\email{vikulovaav@gmail.com}
%\date{}
\thanks{This work is supported by the Russian Science Foundation under grant no.18-11-00121. }
\maketitle

\begin{abstract}
In this paper we find the exact value of the Jordan constant for Cremona group of rank $2$ over all finite fields. During the proof we construct a cubic surface over~$\mathbb{F}_2$ with a regular action of the group $\mathrm{S}_6$ which is the maximal automorphism group of cubic surfaces over~$\mathbb{F}_2.$ Moreover, we prove the uniqueness up to isomorphism of such a cubic surface.
\end{abstract}

 \section{Introduction}
The  Cremona group $\mathrm{Cr}_n(\textbf{F})$ of rank $n$ is a group of birational automorphisms of the projective space $\mathbb{P}^n$ over $\textbf{F}.$ Despite the fact that this group appears naturally, its structure is extremely complicated for $n \geqslant 2.$ Moreover, the description of finite subgroups seems very hard to obtain. Even in the first non-trivial case of rank~$2$ there is a description of conjugacy classes of finite subgroups only over $\mathbb{C}$ (see~\cite{DolgIsk}). Nevertheless, we can understand some properties  of finite subgroups of Cremona groups.

\begin{definition}[{\cite[Definition 2.1]{Popov}}]
A group $G$ is called \textit{Jordan} if there is a constant~$J$ such that any finite subgroup of the group $G$ has a normal abelian subgroup of index at most~$J.$ The minimal such constant $J$ is called the \textit{Jordan constant} of the group $G$ and is denoted by~$J(G).$  
\end{definition}

J.-P. Serre proved in~\cite[Theorem 5.3]{Serre} that the rank $2$ Cremona group~$\mathrm{Cr}_2(\textbf{F})$   over the field $\textbf{F}$ of characteristic zero is Jordan. However, it is not true for algebraically closed field $\textbf{F}$ of characteristic $p>0,$ because the group $\mathrm{Cr}_2(\textbf{F})$ contains simple subgroup $\text{PSL}_2(\mathbb{F}_{p^n})$ whose order grows with $n.$ In the article~\cite{PS} Yu.~Prokhorov and C.~Shramov proved that the Cremona group of rank $2$ over a finite field is Jordan. Moreover, they obtained the following inequalities for Jordan constants.

\begin{Th}[{\cite[Theorem 1.2]{PS}}]\label{ThPS}
  Let $J(\mathrm{Cr}_2(\mathbb{F}_q))$ be the Jordan constant for the Cremona group $\mathrm{Cr}_2(\mathbb{F}_q).$ Then we have
\begin{align*}
&J(\mathrm{Cr}_2(\mathbb{F}_q))=q^3(q^2-1)(q^3-1),  &\text{if $q \notin \{2,4,8\};$}\\
&J(\mathrm{Cr}_2(\mathbb{F}_q)) \leqslant 696\,729\,600,   &\text{if $q \in \{2,4,8\}.$}
\end{align*}

\end{Th}

Our goal is to find the exact value of the Jordan constant for the Cremona group $\mathrm{Cr}_2(\mathbb{F}_q)$ for $q \in \{2,4,8\},$ where $\mathbb{F}_q$ is the field with $q$ elements. For this it is enough to study the groups of biregular automorphisms of del Pezzo surfaces and conic bundles over~$\mathbb{F}_q,$ because every finite subgroup $G$ of $\mathrm{Cr}_2(\mathbb{F}_q)$ is regularized on a $G$-minimal model of a rational surface. According to \mbox{\cite[Corollary 5.3]{PS}} and~\mbox{\cite[Lemma 6.1]{PS}}, the Jordan constant for groups acting biregularly on conic bundles over  $\mathbb{P}^1$ and del Pezzo surfaces of degrees~\mbox{$4 \leqslant d \leqslant 9$} and~$d=2$ is at most~\mbox{$q^3(q^2-1)(q^3-1)$} which is equal to the order of the automorphism group of $\mathbb{P}^2$ over $\mathbb{F}_q.$ Therefore, it is enough to study the group of regular automorphisms of del Pezzo surface of degrees $1$ and $3$ and their normal abelian subgroups.

In this paper we prove the following theorem.

\begin{Th}\label{jordantheorem}
The Jordan constant~$J(\mathrm{Cr}_2(\mathbb{F}_q))$ of the Cremona group $\mathrm{Cr}_2(\mathbb{F}_q)$ is equal to
$$
J(\mathrm{Cr}_2(\mathbb{F}_q)) = \begin{cases}
16\,482\,816, & \text{if $q=8;$}\\
60\,480, & \text{if $q=4;$}\\
720, & \text{if $q=2.$}
\end{cases}
$$
\end{Th}

\begin{Cor}
The Jordan constant~$J(\mathrm{Cr}_2(\mathbb{F}_q))$ for the Cremona group $\mathrm{Cr}_2(\mathbb{F}_q)$ is equal to
$$
J(\mathrm{Cr}_2(\mathbb{F}_q)) = \begin{cases}
|\textnormal{PGL}_3(\mathbb{F}_q)|, & \text{if $q \neq 2;$}\\
|\mathrm{S}_6| > 168=|\textnormal{PGL}_3(\mathbb{F}_2)|, & \text{if $q=2.$}
\end{cases}
$$
\end{Cor}

As a complement to the proof of Theorem \ref{jordantheorem} we prove the following fact about the automorphisms of cubic surface over  $\mathbb{F}_2.$

\begin{Th}\label{th:15S_6}
Let $S$ be a smooth cubic surface in $\mathbb{P}^3$ over $\mathbb{F}_2.$ Then 
$$
|\mathrm{Aut}(S)| \leqslant 720.
$$
\noindent If the equality holds then we have $\mathrm{Aut}(S) \simeq \mathrm{S}_6.$ Moreover, the cubic surface with such automorphism group is unique up to isomorphism.
\end{Th}

\vspace{5mm}

\textbf{Acknowledgment.} The author wish to warmly thank C.~A.~Shramov for suggesting this problem, a lot of discussion and constant attention to this work. Also the author would like to thank A.~S.~Trepalin for the interesting conversations about cubic surfaces and Yu.~G.~Prokhorov for important remarks.

\section{Smooth cubic surfaces}
It is well-known that the automorphism group of a smooth cubic surface in $\mathbb{P}^3$ is a subgroup in the Weyl group~$W(\text{E}_6)$ (see, for instance,~\cite[Corollary 8.2.40]{DolgClass}). Recall that
$$
|W(\text{E}_6)|=2^7 \cdot 3^4 \cdot 5=51\,840.
$$
\noindent According to the classification of the automorphism groups of smooth cubic surfaces over an algebraically closed field (see~\cite[Table 1]{DolgachevDuncan}), we get that the largest automorphism group of a cubic surface over the algebraically closed field $\overline{\mathbb{F}}_2$ is $\text{PSU}_4(\mathbb{F}_2).$ This group is the automorphism group of the Fermat cubic $S \subset \mathbb{P}^3$ (see~\cite[Table~8]{DolgachevDuncan}) which is defined by the equation 
$$
x^3+y^3+z^3+t^3=0.
$$

\noindent Recall that $|\text{PSU}_4(\mathbb{F}_2)|=2^6 \cdot 3^4 \cdot 5=25\,920.$

However, the situation changes if the base field is $\mathbb{F}_2.$ We prove that the largest automorphism group of a smooth cubic surface over $\mathbb{F}_2$ is of order $720.$

\begin{Th}\label{th:maxS_6}
Let $S$ be a smooth cubic surface in $\mathbb{P}^3$ over $\mathbb{F}_2$. Then we have 
$$
|\textnormal{Aut}(S)| \leqslant 720.
$$
\noindent Moreover, if the order of $\mathrm{Aut}(S)$ is equal to $720$, then we have~$\textnormal{Aut}(S) \simeq \mathrm{S}_6.$
\end{Th}

\begin{proof}
The group $\text{Aut}(S)$ is contained in the group $\text{PGL}_4(\mathbb{F}_2),$ because the embedding~\mbox{$S \subset \mathbb{P}^3$} is given by the linear system $|-K_S|$ which is invariant under the automorphism group $\text{Aut}(S).$ On the other hand, the group $\text{Aut}(S)$ is contained in the Weyl group~$W(\text{E}_6).$ Suppose that there is a cubic surface $S$ such that~$720<|\text{Aut}(S)|.$

The maximal subgroup in the group $\text{PGL}_4(\mathbb{F}_2)$  is isomorphic to one of the following groups (see the list of the maximal subgroups of $\text{PGL}_4(\mathbb{F}_2)$ in~\cite[p.~22]{Atlas}):
$$
\mathrm{A}_7,  \; (\mathbb{Z}/2\mathbb{Z})^3 \rtimes   \text{PGL}_3(\mathbb{F}_2), \; \mathrm{S}_6, \; (\mathbb{Z}/2\mathbb{Z})^4 \rtimes   (\mathrm{S}_3 \times \mathrm{S}_3),\;(\mathrm{A}_5 \times \mathbb{Z}/3\mathbb{Z}) \rtimes \mathbb{Z}/2\mathbb{Z}.
$$

\noindent Therefore, we get that~$\text{Aut}(S)$  is either isomorphic to $\text{PGL}_4(\mathbb{F}_2),$ or to a subgroup in  $\mathrm{A}_7,$ or to a subgroup in $(\mathbb{Z}/2\mathbb{Z})^3 \rtimes   \text{PGL}_3(\mathbb{F}_2).$

The first case is impossible, as $|\text{PGL}_4(\mathbb{F}_2)|$ does not divide $|W(\text{E}_6)|.$ Assume that the second case holds, i.e. the order of the group $\text{Aut}(S)$ divides both the order of the group~$\mathrm{A}_7,$ and the order of the group $W(\text{E}_6),$ that is
$$
|\mathrm{Aut}(S)| \leqslant \gcd(|\mathrm{A}_7|,|W(\mathrm{E}_6)|)=2^3 \cdot 3^2 \cdot 5=360<720.
$$
\noindent Thus, this case is also impossible.

Finally, assume that the third case holds. In other words, we obtain
$$
|\text{Aut}(S)| \leqslant \gcd (|(\mathbb{Z}/2\mathbb{Z})^3 \rtimes   \text{PGL}_3(\mathbb{F}_2)|,|W(\text{E}_6)|)=2^6 \cdot 3 =192<720.
$$

\noindent Therefore, the order of the automorphism group of a smooth cubic surface is at most~$720.$

Using~\cite[p. 22]{Atlas}, we get that the only subgroup of order $720$ in $\text{PGL}_4(\mathbb{F}_2)$ is the group $\mathrm{S}_6,$ because it is the maximal subgroup and the other maximal subgroups do not contain subgroups of order $720.$

\end{proof}

Let us construct the cubic surface with automorphism group $\mathrm{S}_6$.

\begin{Example}\label{example720}
 Let us consider the cubic surface $S \subset \mathbb{P}^3$ over $\mathbb{F}_2$ which is defined by the equation
\begin{equation}\label{cubicsmall}
x^2t+y^2z+z^2y+t^2x=0.
\end{equation}

\noindent It is obvious that this equation defines a smooth cubic surface. Moreover, it passes through every point of $\mathbb{P}^3.$ We show that $\mathrm{Aut}(S) \simeq \mathrm{S}_6.$

Let us consider the matrix
$$
\Omega=
\begin{pmatrix}
0 & 0 & 0 & 1\\
0 & 0 & 1 & 0\\
0 & 1 & 0 & 0\\
1 & 0 & 0 & 0
\end{pmatrix}.
$$

\noindent This matrix corresponds to a skew-symmetric bilinear form. Consider the group which fixes this matrix. It consists of the elements~\mbox{$g \in \mathrm{PGL}_4(\mathbb{F}_2)$} such that 
$$
g^{T}\Omega g=\Omega.
$$
\noindent This group is isomorphic to~{$\mathrm{PSp}_4(\mathbb{F}_2).$}

Note that the left-hand side of (\ref{cubicsmall}) is equal to 
$$
(x^2,y^2,z^2,t^2)^{T}\Omega(x,y,z,t).
$$

\noindent So the group $\mathrm{PSp}_4(\mathbb{F}_2)$ is contained in the automorphism group of $S.$ There is a well-known isomorphism $\text{PSp}_4(\mathbb{F}_2) \simeq \mathrm{S}_6$ (see, for instance,~\mbox{\cite[\S 5]{Diedonne}}). Thus, we get the inclusion $\mathrm{S}_6 \subset \text{Aut}(S).$ But by Theorem~\ref{th:maxS_6} the group~$\mathrm{S}_6$ is the maximal possible automorphism group of a smooth cubic surface over $\mathbb{F}_2.$ Therefore, we have the isomorphism~\mbox{$\text{Aut}(S) \simeq \mathrm{S}_6.$}

Note that the cubic surface \eqref{cubicsmall} is rational. Indeed, there are two  skew lines $l_1$ and $l_2$ defined by the equations $x=y=0$ and $z=t=0,$ respectively. This gives us a birational isomorphism $\mathbb{P}^1 \times \mathbb{P}^1 \dashrightarrow S.$

\end{Example}

\begin{Remark}
In the article~\cite{Fatma} by the computer calculations it was shown that the automorphism group of the cubic surface  (\ref{cubicsmall}) is of order $720.$ However, the structure of the automorphism group was not studied.
\end{Remark}

\section{Del Pezzo surfaces of degree $1$}

In this section we consider smooth del Pezzo surfaces of degree $1.$

\begin{Prop}\label{dP1}
Let $S$ be a smooth del Pezzo surface of degree $1$ over $\mathbb{F}_q.$ Then the order of the group $\textnormal{Aut}(S)$ satisfies the inequality
$$
\textnormal{Aut}(S) \leqslant 2q^4(q-1)^2(q+1) < q^3(q^3-1)(q^2-1).
$$

\end{Prop}

\begin{proof}

A degree $2$ regular map
\begin{equation}\label{eq:-2K_S}
\phi_{|-2K_S|}:S \to \mathbb{P}(1,1,2)
\end{equation}
\noindent defined by the linear system $|-2K_S|$ gives us the following exact sequence:
$$
1 \to G \to \text{Aut}(S) \to \text{Aut}(\mathbb{P}(1,1,2)),
$$

\noindent where $|G| \leqslant 2.$ Let us find the automorphism group $\text{Aut}(\mathbb{P}(1,1,2)).$ 
It is obvious (see~\cite[\S 8.3]{autwps}) that the automorphisms of $\mathbb{P}(1,1,2)$ can be written in the following way:
$$
[x:y:z] \mapsto [ax+by:cx+dy:ez+f(x,y)],
$$

\noindent where the matrix
$$
\begin{pmatrix}
a & b\\
c & d
\end{pmatrix}
$$

\noindent is non-degenerate, $e \in \mathbb{F}^*_{q},$ and $f(x,y)$ is a homogeneous polynomial of degree $2$. So we obtain
$$
\text{Aut}(\mathbb{P}(1,1,2)) \simeq \frac{\left(\text{GL}_2(\mathbb{F}_{q}) \times \mathbb{F}_{q}^*\right)\ltimes \left(\mathbb{F}_{q}\right)^3}{\mathbb{F}_{q}^*}.
$$

\noindent It follows that the order of $\text{Aut}(S)$ satisfies the inequality
$$
|\text{Aut}(S)| \leqslant 2\cdot|\text{Aut}(\mathbb{P}(1,1,2))|=2q^4(q-1)^2(q+1) < q^3(q^3-1)(q^2-1).
$$

\end{proof}

\begin{Remark}
It can be shown that the morphism (\ref{eq:-2K_S}) is separable (see, for example,~\cite[Proposition~1.2]{Dolgachev2}). Therefore, in the proof of Proposition~\ref{dP1} the group~$G$ is isomorphic to $\mathbb{Z}/2\mathbb{Z}.$
\end{Remark}

\section{Proof of Theorem \ref{jordantheorem}}
In this section we prove Theorem \ref{jordantheorem}. First of all, let us remind a result from the article \cite{PS}:

\begin{Prop}[{\cite[Corollary 5.3]{PS}}]\label{corformPS}
Let $S$ be a smooth del Pezzo surface of degree not equal to $1$ or $3$ over $\mathbb{F}_q.$ Then the group $\mathrm{Aut}(S)$ contains a normal abelian subgroups of index at most $q^3(q^2-1)(q^3-1).$
\end{Prop}

Now we prove our main theorem.
\begin{proof}[Proof of Theorem \ref{jordantheorem}]

As it was mentioned above, any finite subgroup $G$ of the Cremona group $\mathrm{Cr}_2(\mathbb{F}_q)$ is regularized on a $G$-minimal model of a rational surface. Thus, for all automorphism groups of del Pezzo surfaces and conic bundles over~$\mathbb{P}^1$ it is enough to find the minimal index of a normal abelian subgroup.  The maximal among such numbers is the required Jordan constant.

According to Proposition \ref{corformPS}, a normal abelian subgroup in the automorphism group of a del Pezzo surface of degree not equal to $1$ or $3$ over $\mathbb{F}_q$ has index at most~$q^3(q^2-1)(q^3-1).$ And according to~\cite[Lemma 6.1]{PS}, a normal abelian subgroup in the automorphism group of a conic bundle over $\mathbb{P}^1$  also has index at most~$q^3(q^2-1)(q^3-1).$ By Proposition \ref{dP1} the automorphism group of a smooth del Pezzo surface of degree $1$ over $\mathbb{F}_q$ is of order at most~$q^3(q^2-1)(q^3-1).$

It is enough to find the maximum of all minimal indexes of normal abelian subgroups in every automorphism group for smooth cubic surfaces. If $q=4\; \text{or}\;8$ then 
$$
|W(\text{E}_6)|=51\,840<60\,480=|\text{PGL}_3(\mathbb{F}_4)|\leqslant|\text{PGL}_3(\mathbb{F}_q)|,
$$
\noindent therefore, the order of the automorphism group of a smooth cubic surface is less than~\mbox{$q^3(q^2-1)(q^3-1).$} And if $q=2$ then by Theorem \ref{th:maxS_6} the order of the automorphism group of a smooth cubic surface is at most
$$
720>168=q^3(q^2-1)(q^3-1).
$$

\noindent All these values, i.e. $720$ for $\mathbb{F}_2$ and $q^3(q^2-1)(q^3-1)$ for $\mathbb{F}_4$ and $\mathbb{F}_8$ are realized as orders of the automorphism groups of some rational surfaces. Indeed, for~\mbox{$q=4\; \text{and}\;8$} these values are attained on the automorphism group of $\mathbb{P}^2.$ This group is isomorphic to  
$\text{PGL}_3(\mathbb{F}_q)$ and it does not have non-trivial normal abelian subgroups (see~~\cite[Lemma 2.4]{PS}). For $q=2$ the value of the Jordan constant is attained on the rational cubic surface with the automorphism group isomorphic to~$\mathrm{S}_6$ (see Example~\ref{example720}). This group does not have non-trivial normal abelian subgroups either.

\end{proof}

\appendix
\section{Uniqueness of the cubic surface with the automorphism group~$\mathrm{S}_6$ over $\mathbb{F}_2$}

In this section we prove Theorem \ref{th:15S_6}.

\begin{Lemma}\label{lemma:15points}
Let $S$ be a smooth cubic surface in the projective space $\mathbb{P}^3$ over $\mathbb{F}_2,$ such that it passes through all $15$ points in $\mathbb{P}^3.$ Then $S$ is isomorphic to a cubic surface~(\ref{cubicsmall}) and its automorphism group is  $\mathrm{S}_6.$
\end{Lemma}

\begin{proof}
Note that the cubic surface passing through all points in $\mathbb{P}^3$ can be defined by the following equation:
$$
a_1xy(x+y)+a_2xz(x+z)+a_3xt(x+t)+a_4yz(y+z)+a_5yt(y+t)+a_6zt(z+t)=0,
$$

\noindent where $a_i \in \mathbb{F}_2.$ Thus, there are $63$ such cubic surface, not necessarily smooth.

Let us consider the action of the group $\mathrm{PGL}_4(\mathbb{F}_2)$ on smooth cubic surfaces in~$\mathbb{P}^3$ over $\mathbb{F}_2,$ which pass through $15$ points. Denote by $\mathrm{Orb}(S)$ the orbit of  $S$ under the action of $\mathrm{S}_6.$ We have 
$$
\frac{|\mathrm{PGL}_4(\mathbb{F}_2)|}{|\mathrm{Aut}(S)|}=|\mathrm{Orb}(S)|.
$$

\noindent Let us find the number of elements in $\mathrm{Orb}(S).$ Consider a singular cubic surface which is a union of three hyperplanes in $\mathbb{P}^3$ passing through one line. It is obvious that  there are $15$ points on such a cubic surface. It is not hard to see that there are exactly~$35$ lines in $\mathbb{P}^3$ over $\mathbb{F}_2.$ It follows that there are at least $35$ such singular cubic surfaces. Thus, there are at most $28$ smooth cubic surfaces passing through~$15$ points. Therefore, we obtain $|\mathrm{Orb}(S)| \leqslant 28.$ So we have the inequality
$$
|\mathrm{Aut}(S)| \geqslant 720.
$$
According to Theorem \ref{th:maxS_6}, this inequality holds true if and only if  $\mathrm{Aut}(S)=\mathrm{S}_6$ and~\mbox{$|\mathrm{Orb}(S)| = 28.$} In other words, all smooth cubic surfaces passing through $15$ points are isomorphic to each other. As the cubic surface of the form (\ref{cubicsmall}) passes through $15$ points, we get that $S$ is isomorphic to the cubic surface defined by the equation~(\ref{cubicsmall}). 

\end{proof}

\begin{Lemma}\label{lemma:cubicsmall}
Assume that the group $\mathrm{S}_6$ is the automorphism group of a smooth cubic surface~$S$ over $\mathbb{F}_2.$ Then $S$ is isomorphic to a cubic surface defined by the equation~(\ref{cubicsmall}).
\end{Lemma}

\begin{proof}
Note that according to the Chevalley--Warning theorem, there is a point on the cubic surface $S.$ Denote by $p$  one of such points. The action of the group~$\mathrm{S}_6$ on $S$ defines its action on $\mathbb{P}^3.$ Assume that the length of the orbit of the point~$p$ is~$l \neq 5,10,15.$ So the stabilizer of every point in the orbit is a subgroup of $\mathrm{S}_6$ of index $l.$ Moreover, the stabilizer contains a subgroup $\mathbb{Z}/5\mathbb{Z},$ because $l$ and $5$ are coprime.  So the group~$\mathbb{Z}/5\mathbb{Z}$ acts non-trivially on the tangent space to  $\mathbb{P}^3$ at the point $p$ (see, for instance,~\cite[Theorem 3.7]{ChenShramov}). This means that~$\mathbb{Z}/5\mathbb{Z} \subset  \mathrm{GL}_3(\mathbb{F}_2).$ But it is impossible, because  $|\mathrm{GL}_3(\mathbb{F}_2)|=168.$

The case $l=5$ is impossible as well, because the group $\mathrm{S}_6$ does not contain subgroups of index $5.$  If $l=10$ then there are orbits of the action of $\mathrm{S}_6$ on $\mathbb{P}^3$ of length at most $5.$ But as we have just shown, it is also impossible.

Therefore, the only possible case is $l=15.$ Thus, the cubic surface $S$ with an action of $\mathrm{S}_6$ passes through $15$ points. By Lemma \ref{lemma:15points} such cubic surfaces are isomorphic to the cubic surface of the form~(\ref{cubicsmall}).

\end{proof}

Now we prove Theorem \ref{th:15S_6}.

\begin{proof}[Proof of Theorem \ref{th:15S_6}]
The first part of the theorem follows from Theorem \ref{th:maxS_6}. The existence and uniqueness follow from Lemmas \ref{lemma:15points}~and~\ref{lemma:cubicsmall}.

\end{proof}

Also we obtain the following corollary from Lemmas \ref{lemma:15points} and \ref{lemma:cubicsmall}.

\begin{Cor}
The group $\mathrm{S}_6$ acts on a smooth cubic surface $S$ over $\mathbb{F}_2$ if and only if $S$ passes through all $15$ points in $\mathbb{P}^3.$ 
\end{Cor}

%
%--------------%
% BIBLIOGRAPHY %
%--------------%
%
%\bibliography{bib}
%\bibliographystyle{amsplain}
%

\providecommand{\bysame}{\leavevmode\hbox to3em{\hrulefill}\thinspace}
\providecommand{\MR}{\relax\ifhmode\unskip\space\fi MR }
\providecommand{\MRhref}[2]{%
  \href{http://www.ams.org/mathscinet-getitem?mr=#1}{#2}
}
\providecommand{\href}[2]{#2}

 \end{document}